\documentclass[11pt]{article}
\usepackage[T1]{fontenc}
\usepackage{lmodern}
\usepackage{amsmath,amssymb,amsthm,mathtools}
\usepackage{booktabs,array,longtable,float}
\usepackage[a4paper,margin=25mm]{geometry}
\usepackage{microtype}
\usepackage{xcolor}
\definecolor{inkblue}{RGB}{25,68,88}
\usepackage[colorlinks=true,breaklinks=true,linkcolor=inkblue,citecolor=inkblue,urlcolor=inkblue]{hyperref}
\hypersetup{pdftitle={Nonexistence of Strongly Regular Graphs via Gegenbauer Positivity and SRG Counting},pdfauthor={Zhen-Qi Liao, Jiaxin Wang, Wei-Hsuan Yu, Jianfeng Cai}}
\newtheorem{theorem}{Theorem}[section]
\newtheorem{proposition}[theorem]{Proposition}
\newtheorem{lemma}[theorem]{Lemma}

\theoremstyle{remark}

\newcommand{\R}{\mathbb R}
\newcommand{\Q}{\mathbb Q}
\newcommand{\Z}{\mathbb Z}

\DeclareMathOperator{\Adj}{Adj}
\DeclareMathOperator{\Aut}{Aut}
\DeclareMathOperator{\rank}{rank}
\DeclareMathOperator{\tr}{tr}

\numberwithin{equation}{section}
\allowdisplaybreaks[2]
\title{Nonexistence of Strongly Regular Graphs\\
via Gegenbauer Positivity and SRG Counting}
\author{%
\begin{tabular}{cc}
Zhen-Qi Liao$^{1}$ & Jiaxin Wang$^{2}$\\[2pt]
\small\href{mailto:l0911487082@gmail.com}{l0911487082@gmail.com} &
\small\href{mailto:jwangka@connect.ust.hk}{jwangka@connect.ust.hk}\\[7pt]
Wei-Hsuan Yu$^{1}$ & Jianfeng Cai$^{2}$\\[2pt]
\small\href{mailto:whyu@math.ncu.edu.tw}{whyu@math.ncu.edu.tw} &
\small\href{mailto:jfcai@ust.hk}{jfcai@ust.hk}
\end{tabular}\\[8pt]
\small $^{1}$Department of Mathematics, National Central University, Taiwan\\[3pt]
\small $^{2}$Department of Mathematics,\\[-2pt]
\small The Hong Kong University of Science and Technology, Hong Kong}
\date{14 October 2026}
\begin{document}
\maketitle
\begin{abstract}
This paper develops a natural extension of the multipoint spherical
semidefinite method of Liao and Yu~\cite{LY}, studying the interaction
between semidefinite programming and strongly regular graph counting.
Starting from the positive semidefinite kernels furnished by Gegenbauer
polynomials, we express weighted kernel sums and counting identities
in the same labelled induced-subgraph counts. Eliminating the counting
equations gives a semidefinite feasibility problem. Numerical dual
solutions guide the construction of rational certificates, whose
contradiction identities are verified exactly. Our aim is a
systematic method that treats many parameter sets through the same
construction. For least eigenvalue $-3$, it gives new counting
certificates for
$(351,140,73,44)$, $(550,162,75,36)$, $(703,182,81,35)$ and
$(1344,221,88,26)$, previously excluded by Liao and
Yu~\cite[Theorem~1]{LY}, and for six additional parameter sets:
\[
\begin{array}{@{}lll@{}}
(375,102,45,21), & (405,132,63,33), & (441,88,35,13),\\
(476,133,60,28), & (540,147,66,30), & (575,112,45,16).
\end{array}
\]
Seven certificates use Gegenbauer positivity
and counting information alone at the displayed orders. The cases
$v=375,441,575$ additionally use proved short-vector restrictions,
unique vertex extensions, and integer common-neighbour inequalities.
The proofs end in displayed inequalities with incompatible signs.
Their justification is given by the dual identities and the proved
positivity and counting relations; accompanying code automates the
longer rational arithmetic.
\end{abstract}

\section{Introduction}\label{sec:intro}
A strongly regular graph with parameters $(v,k,\lambda,\mu)$ is a
$k$-regular graph on $v$ vertices in which adjacent vertices have
$\lambda$ common neighbours and distinct nonadjacent vertices have
$\mu$ common neighbours. Its eigenspaces provide spherical two-distance
sets, whose sizes can be bounded by semidefinite programming.
The multipoint spherical approach of Liao and Yu~\cite{LY} obtains
nonexistence results by bounding the size of such sets. Here our goal
also includes understanding how semidefinite programming interacts with
the counting structure of strongly regular graphs. We start directly
from Gegenbauer positivity and formulate feasibility in induced-subgraph
counts. The eigenspace embedding supplies the inner products at which
the kernels are evaluated; a cardinality bound for general spherical
two-distance sets is no longer the optimization problem. This formulation
retains the degree and common-neighbour information of the graph.

Kov\'a\v{c}ikov\'a already combined Gegenbauer inequalities with
induced-subgraph counts~\cite[Proposition~4.2, Lemma~4.2, and
Sections~4.2--4.3]{kovacikova}. Her construction represents induced
subgraphs by normalized nonzero sums of their spherical vertex vectors and
groups the resulting sums by induced union types. Here we use Musin's
projected kernels at fixed reference configurations, with arbitrary
weights indexed by external adjacency profiles. The resulting positive
semidefinite matrices are linear in labelled induced-subgraph counts.
We combine these matrices with the SRG extension equations and, where
needed, further proved structural relations to construct explicit
rational nonexistence certificates.

Our four-point experiments in late 2025 combined the counting relations
of~\cite[Proposition~3.2 and Section~3.3.3]{kovacikova} with semidefinite
constraints but excluded none of the tested parameter sets. The
certificates below use counts through five, six, or seven vertices.

Musin's positivity theorem~\cite{musin} supplies inequalities in the
distribution of small induced subgraphs, while double counting supplies
linear equations in the same variables. Eliminating the latter gives
a semidefinite problem whose dual guides the search for rational
contradiction identities.
Our aim is a reusable method that treats many parameter sets through
one construction. The counting identities, kernel inequalities, and
dual certificate mechanism are shared across the applications; further
structural information enters through the same count variables.
The certificate takes the form of a rational identity in graph counts,
and each case ends in an explicit contradictory inequality. The
displayed formulas and the proved relations establish the argument;
the verification programs check its finite coefficient calculations.
For example, the reported six-point spherical upper bound for
$(540,147,66,30)$ is $548$~\cite[Table~3]{LY}; here five-point counts,
three reference vertices, and degrees $1,2$ give a contradiction.

The focus on least eigenvalue $-3$ is guided both by the candidate list
of Koolen and Gebremichel~\cite[Table~1]{KG} and by our repeated computational
experiments with spherical bounds and SRG counting. Their 2022
Table~1 lists twelve unresolved candidates among primitive strongly
regular graphs with least eigenvalue $-3$, $v>276$, and
$\mu\notin\{6,9\}$. Liao and Yu~\cite[Theorem~1]{LY} subsequently
excluded the four rows marked $\dagger$ in Table~\ref{tab:parameters}.
The ten parameter sets proved here are listed in that table; our aim is
to exhibit their nonexistence within a common semidefinite framework.

\begin{table}[htbp]
\centering\small
\caption{The ten excluded parameter sets. Here $r$ has multiplicity $d$,
and $p$ is the largest graph order in the certificate.
$\dagger$ marks the four cases previously treated by Liao--Yu~\cite{LY}.
Every certificate combines Gegenbauer positivity with strongly regular
graph counting information. The final column records whether the
additional structural information proved in Section~\ref{sec:structural}
is used.}
\label{tab:parameters}
\begin{tabular}{@{}rrrrrrcl@{}}
\toprule
$v$&$k$&$\lambda$&$\mu$&$r$&$d$&$p$&Additional structural information\\
\midrule
$351^\dagger$&140&73&44&32&26&6&Not needed\\
375&102&45&21&27&34&7&Used (Section~\ref{sec:structural})\\
405&132&63&33&33&30&7&Not needed\\
441&88&35&13&25&44&7&Used (Section~\ref{sec:structural})\\
476&133&60&28&35&34&6&Not needed\\
540&147&66&30&39&35&5&Not needed\\
$550^\dagger$&162&75&36&42&33&5&Not needed\\
575&112&45&16&32&46&7&Used (Section~\ref{sec:structural})\\
$703^\dagger$&182&81&35&49&37&5&Not needed\\
$1344^\dagger$&221&88&26&65&56&5&Not needed\\
\bottomrule
\end{tabular}
\end{table}

\begin{theorem}\label{thm:main}
No strongly regular graph exists for any parameter set in
Table~\ref{tab:parameters}.
\end{theorem}
Seven cases are proved in Section~\ref{sec:direct} from Gegenbauer
positivity and counting information. For $375$, $441$, and $575$,
Section~\ref{sec:structural} adds structural information to the same
method. The displayed orders record the available certificates;
minimal orders are not asserted.

\paragraph{Data and code availability.}
The article, supplementary material, exact certificates, and verification
programs are available from the
\href{https://github.com/SelenaJXWANG/srg-gegenbauer-counting}{public companion repository}.
The certificate data used in the proofs are identified by the fixed
repository snapshot \href{https://github.com/SelenaJXWANG/srg-gegenbauer-counting/tree/85cd39b74c9ef8a61ab746deb1102af3eeae7e10}{\texttt{85cd39b}}~\cite{certificateData}.
Appendix~\ref{app:delivery} gives the certificate locations and the
command for reproducing their exact arithmetic checks.

Section~\ref{sec:method} develops the construction and its Lagrange
dual. Sections~\ref{sec:direct} and~\ref{sec:structural} give the
contradictions. Detailed rational coefficients and verification
conventions are collected in the appendices and the separate
$405$ certificate supplement~\cite{supp405}. Appendix~\ref{app:delivery}
identifies the accompanying certificate snapshot, its exact inputs,
and the common verification command.

\section{Gegenbauer positivity and counting information}\label{sec:method}
We express weighted Gegenbauer kernel sums in labelled induced-subgraph
counts and derive exact rational nonexistence certificates.

\subsection{Vertices, spherical vectors, and fixed reference vertices}
We use the normalized eigenspace embedding in
\cite[Section~4.1]{LY}, specialized to least eigenvalue $-3$.
Let $\Gamma=(V,E)$ be a primitive strongly regular graph with parameters
$(v,k,\lambda,\mu)$, adjacency matrix $\Adj(\Gamma)$, and nontrivial
eigenvalues $r,-3$, with $r$ of multiplicity $d$. All parameter sets considered
satisfy $\mu>0$ and $v-2k+\lambda>0$, which force the graph and its
complement to be connected. Then
\begin{equation}\label{eq:spectral}
 r=\lambda-\mu+3,\qquad k=3r+\mu,\qquad
 d=\frac{3(v-1)-k}{r+3}.
\end{equation}
These formulas follow from the quadratic equation for the nontrivial
eigenvalues and $\tr\Adj(\Gamma)=0$. The orthogonal projector onto the
$r$-eigenspace is
\[
 E_r=\frac{1}{r+3}\left(\Adj(\Gamma)+3I-\frac{k+3}{v}J\right).
\]
Here $I$ and $J$ are the $v\times v$ identity and all-one matrices.
The projector has diagonal $d/v$. Normalizing its columns gives unit vectors
$p_x\in S^{d-1}\subset\R^d$, one for each $x\in V$, with
\begin{equation}\label{eq:sphere}
 \langle p_x,p_y\rangle=
 \begin{cases}1,&x=y,\\ \beta,&x\sim_\Gamma y,\\
 \alpha,&x\ne y,\ x\not\sim_\Gamma y,
 \end{cases}
 \qquad
 \alpha=-\frac{k+3}{3v-k-3},\quad
 \beta=\frac{v-k-3}{3v-k-3}.
\end{equation}
Equivalently, $\beta=r/k$ and $\alpha=-(r+1)/(v-k-1)$.

Fix distinct vertices $r_0,\ldots,r_{m-1}$ whose vectors are linearly
independent, with $1\le m\le d-2$, and put
\[
 W=\operatorname{span}\{p_{r_0},\ldots,p_{r_{m-1}}\},\qquad
 D=d-m,\qquad X=V\setminus\{r_0,\ldots,r_{m-1}\}.
\]
For $x\in X$, write $p_x=p_x^\parallel+p_x^\perp$ for its orthogonal
decomposition in $W\oplus W^\perp$.

\subsection{The scalar kernel and its positivity}\label{subsec:kernel}
For $\ell\in\Z_{\ge0}$, let $P_\ell^{(D)}$ be the Gegenbauer polynomial
for $S^{D-1}$, normalized by $P_\ell^{(D)}(1)=1$. For $D=2$ use
$P_\ell^{(2)}(\cos\theta)=\cos(\ell\theta)$. Define
\begin{equation}\label{eq:kernel}
 K_{\ell,W}(x,y)=
 (\|p_x^\perp\|\|p_y^\perp\|)^\ell
 P_\ell^{(D)}\!\left(
 \frac{\langle p_x^\perp,p_y^\perp\rangle}
 {\|p_x^\perp\|\|p_y^\perp\|}\right).
\end{equation}
At a zero orthogonal component use its polynomial extension:
$K_{0,W}=1$, and $K_{\ell,W}(x,y)=0$ for $\ell\ge1$ if either
component is zero. Musin's construction and positivity theorem
\cite[Definition~2.1 and Theorem~3.1]{musin}
give
\begin{equation}\label{eq:positivity}
 \sum_{x,y\in X}a_xa_yK_{\ell,W}(x,y)\ge0
 \qquad\bigl((a_x)_{x\in X}\in\R^X\bigr).
\end{equation}
Thus Gegenbauer positivity means that the kernel matrix
$(K_{\ell,W}(x,y))_{x,y\in X}$ is positive semidefinite.
Every double sum uses ordered pairs and includes $x=y$.
Let $u_x\in\R^m$ be the coordinates of $p_x^\parallel$ in an
orthonormal basis of $W$, and put $t_{xy}=\langle p_x,p_y\rangle$.
Musin's variables then satisfy
$1-\|u_x\|^2=\|p_x^\perp\|^2$ and
$t_{xy}-u_x^Tu_y=\langle p_x^\perp,p_y^\perp\rangle$;
this gives exactly \eqref{eq:kernel}.

In degrees zero, one, and two the kernels are
\begin{equation}\label{eq:lowkernels}
 K_{0,W}=1,\quad K_{1,W}(x,y)=\langle p_x^\perp,p_y^\perp\rangle,
 \quad K_{2,W}(x,y)=
 \frac{D\langle p_x^\perp,p_y^\perp\rangle^2
 -\|p_x^\perp\|^2\|p_y^\perp\|^2}{D-1}.
\end{equation}
Writing $q_x=p_x^\perp$ only in the next display, their sums are
\begin{align}
 \sum_{x,y}a_xa_yK_{0,W}(x,y)&=\left(\sum_xa_x\right)^2,\nonumber\\
 \sum_{x,y}a_xa_yK_{1,W}(x,y)&=\left\|\sum_xa_xq_x\right\|^2,\label{eq:squares}\\
 \sum_{x,y}a_xa_yK_{2,W}(x,y)&=
 \frac{D}{D-1}\left\|\sum_xa_x
 \left(q_xq_x^T-\frac{\|q_x\|^2}{D}I_D\right)\right\|_F^2.\nonumber
\end{align}
The last equality follows from
$\langle uu^T,vv^T\rangle_F=\langle u,v\rangle^2$ by taking trace-free
parts. Thus these three degrees also have a direct proof of positivity.

\subsection{Labelled counts and one-vertex extensions}\label{subsec:counts}
All graphs are finite and simple. A labelled $t$-vertex graph $F$ has
vertex set $[t]=\{0,\ldots,t-1\}$. Define
\begin{equation}\label{eq:NF}
 N_F=\#\{\theta:[t]\hookrightarrow V:
 i\sim_Fj\ \Longleftrightarrow\ \theta(i)\sim_\Gamma\theta(j)
 \text{ for }i\ne j\}.
\end{equation}
A hooked arrow between graphs denotes an induced embedding. In particular,
$N_{K_1}=v$, $N_{K_2}=vk$, and $N_{K_3}=vk\lambda$.
Let $\mathcal G_t$ contain one labelled representative of each isomorphism
class on $[t]$. Isomorphic graphs have equal counts, and each vertex subset
inducing $F$ contributes $|\Aut F|$ embeddings. Hence
\begin{equation}\label{eq:totals}
 \sum_{F\in\mathcal G_t}\frac{N_F}{|\Aut F|}=\binom vt.
\end{equation}

For $H$ on $[h]$ and $B\subseteq[h]$, let $H_B$ be the graph on
$[h+1]$ obtained by making the new vertex $h$ adjacent exactly to $B$.
For fixed $B_0\subseteq[h]$ with $|B_0|\le2$, put
\begin{equation}\label{eq:extensionfactor}
 e(H,B_0)=\begin{cases}
 v-h,&B_0=\varnothing,\\
 k-\deg_H(i),&B_0=\{i\},\\
 \lambda-|\mathcal N_H(i)\cap \mathcal N_H(j)|,&B_0=\{i,j\},\ i\sim_Hj,\\
 \mu-|\mathcal N_H(i)\cap \mathcal N_H(j)|,&B_0=\{i,j\},\ i\not\sim_Hj.
 \end{cases}
\end{equation}
Here $\mathcal N_H(i)$ denotes the neighbourhood of $i$ in $H$.
\begin{proposition}\label{prop:extensions}
For every $H,B_0$ as above,
\begin{equation}\label{eq:extensions}
 e(H,B_0)N_H=\sum_{B_0\subseteq B\subseteq[h]}N_{H_B}.
\end{equation}
\end{proposition}
\begin{proof}
Count pairs $(\varphi,x)$ where $\varphi:H\hookrightarrow\Gamma$,
$x\notin\varphi([h])$, and $x$ is adjacent to every $\varphi(i)$
with $i\in B_0$. For each $\varphi$ there are $e(H,B_0)$ choices.
Alternatively set $B=\{i:x\sim_\Gamma\varphi(i)\}$. Extending
$\varphi$ by $h\mapsto x$ gives an induced embedding of exactly one
$H_B$, with $B\supseteq B_0$. Restriction reverses this construction.
\end{proof}
Isomorphic extensions occur with their full multiplicities; no automorphism
factor appears. This is the labelled form of the extension counts in
\cite[Proposition~3.2(i)--(iv)]{kovacikova}. For example,
\[
 (\lambda-1)N_{K_3}=N_{K_4-e}+N_{K_4}.
\]

\subsection{Weights indexed by adjacency sets}\label{subsec:weights}
Fix $H$ on $[m]$, where $1\le m\le d-2$, and form its Gram matrix
\begin{equation}\label{eq:GH}
 G_H=(1-\alpha)I_m+\alpha J_m+(\beta-\alpha)\Adj(H).
\end{equation}
Assume $G_H\succ0$, so every induced copy of $H$ has independent
spherical vectors.
For $\varphi:H\hookrightarrow\Gamma$, set
\[
 W_\varphi=\operatorname{span}\{p_{\varphi(0)},\ldots,p_{\varphi(m-1)}\},
 \qquad K_\ell^\varphi=K_{\ell,W_\varphi}.
\]
For an induced embedding $\iota:H\hookrightarrow G$, define
\begin{equation}\label{eq:sigma}
 \sigma_\iota(z)=\{i\in[m]:z\sim_G\iota(i)\},
 \qquad z\in V(G)\setminus\iota([m]).
\end{equation}
Let $\mathbf1_B$ be the characteristic column of $B\subseteq[m]$.
Order $2^{[m]}$ lexicographically by these columns. For $m=3$ the order is
\begin{equation}\label{eq:profileorder}
 \varnothing,\{2\},\{1\},\{1,2\},\{0\},\{0,2\},\{0,1\},\{0,1,2\}.
\end{equation}
Choose any $f:2^{[m]}\to\R$, independently of $\varphi$, and identify
it with its column of $2^m$ values. Apply \eqref{eq:positivity} with
$a_x=f(\sigma_\varphi(x))$ and then sum over $\varphi$:
\begin{equation}\label{eq:S}
 S(H,\ell,f):=\sum_{\varphi:H\hookrightarrow\Gamma}
 \sum_{x,y\in V\setminus\varphi([m])}
 f(\sigma_\varphi(x))f(\sigma_\varphi(y))K_\ell^\varphi(x,y)\ge0.
\end{equation}
There is no division by $|\Aut H|$; an empty sum is zero.
Each summand has support $\varphi([m])\cup\{x,y\}$, of size $m+1$
if $x=y$ and $m+2$ otherwise.

\subsection{Exact coefficients on abstract graphs}\label{subsec:coefficients}
Put
\[
 b(B)=\alpha\mathbf1+(\beta-\alpha)\mathbf1_B\in\R^m,
\]
where $\mathbf1$ is the all-one column.
For $\psi:H\hookrightarrow F$ and $a,b\in V(F)\setminus\psi([m])$,
write $b_a^\psi=b(\sigma_\psi(a))$ and let $t_{ab}$ be $1,\beta,\alpha$
according as $a=b$, $a\sim_F b$, or $a\ne b$ and $a\not\sim_F b$.
If
$P_\ell^{(D)}(t)=\sum_{j=0}^{\lfloor\ell/2\rfloor}
\gamma_{\ell,j}^{(D)}t^{\ell-2j}$, define
\begin{align}
 K_\ell^\psi(a,b)=\sum_{j=0}^{\lfloor\ell/2\rfloor}
 \gamma_{\ell,j}^{(D)}
 \bigl(t_{ab}-(b_a^\psi)^TG_H^{-1}b_b^\psi\bigr)^{\ell-2j}
 \bigl[(1-(b_a^\psi)^TG_H^{-1}b_a^\psi)
 (1-(b_b^\psi)^TG_H^{-1}b_b^\psi)\bigr]^j.
 \label{eq:formal}
\end{align}
This polynomial is defined even for unrealizable Gram data. For an actual
copy, orthogonal projection gives
\[
 \|p_x^\perp\|^2=1-b(\sigma_\varphi(x))^TG_H^{-1}b(\sigma_\varphi(x)),
 \quad
 \langle p_x^\perp,p_y^\perp\rangle
 =\langle p_x,p_y\rangle-b(\sigma_\varphi(x))^TG_H^{-1}b(\sigma_\varphi(y)),
\]
so \eqref{eq:formal} agrees with \eqref{eq:kernel}; see
\cite[Corollary~3.2]{musin} and \cite[Definition~2.11]{kaoyu}.
Positivity applies to the sums on actual
vectors, not to individual formal coefficients.

Define
\begin{equation}\label{eq:c}
 c_{H,\ell,f}(F)=\frac{1}{|\Aut F|}
 \sum_{\psi:H\hookrightarrow F}
 \sum_{\substack{a,b\in V(F)\setminus\psi([m])\\
 \psi([m])\cup\{a,b\}=V(F)}}
 f(\sigma_\psi(a))f(\sigma_\psi(b))K_\ell^\psi(a,b).
\end{equation}
The inner sum uses the diagonal when $|F|=m+1$ and both ordered pairs
of the two exterior vertices when $|F|=m+2$.
\begin{proposition}\label{prop:countingkernel}
For the actual labelled counts,
\begin{equation}\label{eq:countingkernel}
 S(H,\ell,f)=\sum_{F\in\mathcal G_{m+1}\cup\mathcal G_{m+2}}
 c_{H,\ell,f}(F)N_F\ge0.
\end{equation}
\end{proposition}
\begin{proof}
Group the summands of \eqref{eq:S} by their support $U$. If
$\Gamma[U]\cong F$, choose one isomorphism $\theta:F\to\Gamma[U]$.
Then $(\psi,a,b)\mapsto(\theta\circ\psi,\theta(a),\theta(b))$
is a bijection between the triples in \eqref{eq:c} and those supported
on $U$ in \eqref{eq:S}. It preserves the profiles and kernels, without
requiring invariance of $f$. Each $U$ contributes
$|\Aut F|c_{H,\ell,f}(F)$, and there are $N_F/|\Aut F|$ such subsets.
\end{proof}
\subsection{The semidefinite model and counting elimination}\label{subsec:sdp}
Choose an integer counting order $1\le p\le v$. Let
$\mathcal G_{\le p}=\bigcup_{t=1}^p\mathcal G_t$ and
$N=(N_F)_{F\in\mathcal G_{\le p}}$.
Collect \eqref{eq:extensions} for $1\le h<p$, $|B_0|\le2$ and
\eqref{eq:totals} for $1\le t\le p$ into
\begin{equation}\label{eq:AN}
 AN=b,\qquad
 \mathcal P=\{N\in\R^{\mathcal G_{\le p}}:AN=b,\ N_F\ge0\}.
\end{equation}
Here $b$ is the right-hand-side column of the counting equations.
The matrix $A$ has one column per graph type, with isomorphic extensions
combined. Every actual SRG supplies a point of $\mathcal P$; integrality
is relaxed. By \eqref{eq:totals},
\begin{equation}\label{eq:countbound}
 0\le N_F\le |\Aut F|\binom{v}{|F|},
\end{equation}
so $\mathcal P$ is compact, possibly empty.

Choose finitely many pairs $(H,\ell)$ with $H\in\mathcal G_m$,
$1\le m\le\min(p-2,d-2)$, $G_H\succ0$, and $\ell\in\Z_{\ge0}$.
Since \eqref{eq:c} is quadratic in $f$,
there is a unique symmetric $2^m\times2^m$ matrix $C_{H,\ell}(F)$ such that
\[
 f^TC_{H,\ell}(F)f=c_{H,\ell,f}(F).
\]
Set $C_{H,\ell}(F)=0$ at other graph orders, and put
\begin{equation}\label{eq:M}
 M_{H,\ell}(N)=\sum_F N_FC_{H,\ell}(F),\qquad
 S(H,\ell,f;N)=f^TM_{H,\ell}(N)f.
\end{equation}
For actual counts, $S(H,\ell,f;N)=S(H,\ell,f)$; the semicolon notation
extends the same linear expression to formal count vectors.
This adjacency-profile indexing is the two-distance, exterior-point
form of the label-indexed multipoint constraints in
\cite[Theorem~2.13]{kaoyu}. All weights at once give
\begin{equation}\label{eq:primal}
 \begin{array}{ll}
 \text{minimize}_{N}&0\\
 \text{subject to}&AN=b,\quad N_F\ge0\quad(F\in\mathcal G_{\le p}),\\
 &M_{H,\ell}(N)\succeq0\quad\text{for each selected }(H,\ell).
 \end{array}
\end{equation}
Every hypothetical graph gives a feasible point, so infeasibility excludes
the parameters. A dual certificate supplies a fixed finite family of
weights giving the contradiction.

\paragraph{Eliminating the counting equations.}\label{subsec:elimination}
Row-reduce $[A\mid b]$, retaining any inconsistent row. If inconsistent,
counting alone gives the contradiction. Otherwise let
$q=|\mathcal G_{\le p}|-\rank A$ and choose the non-pivot counts
$N_{F_1},\ldots,N_{F_q}$ as coordinates $\xi_1,\ldots,\xi_q$.
Then all solutions, and only those solutions, have the form
\begin{equation}\label{eq:affine}
 N_F(\xi)=a_{F0}+\sum_{j=1}^q a_{Fj}\xi_j,
 \qquad N_{F_j}(\xi)=\xi_j,\qquad \xi\in\R^q.
\end{equation}
Indeed, after ordering pivot columns first, the reduced system is
$(I\ R)(N_{\rm piv},N_{\rm free})^T=c$; thus
$N_{\rm piv}=c-R\xi$ and $N_{\rm free}=\xi$.
Abbreviate $M_{H,\ell}(N(\xi))$ by $M_{H,\ell}(\xi)$, and, for
$0\le j\le q$, set
\begin{equation}\label{eq:Mj}
 M_{H,\ell}^{(j)}=\sum_Fa_{Fj}C_{H,\ell}(F),\qquad
 M_{H,\ell}(\xi)=M_{H,\ell}^{(0)}+
 \sum_{j=1}^q\xi_jM_{H,\ell}^{(j)}.
\end{equation}
The reduced primal minimizes $0$ over $\xi\in\R^q$ subject to
$N_F(\xi)\ge0$ for \emph{every} $F$ and
$M_{H,\ell}(\xi)\succeq0$ for every selected pair.

\subsection{The Lagrange dual and exact certificates}\label{subsec:dual}
Use the convention of Boyd and Vandenberghe
\cite[Sections~5.8.1 and~5.9.1]{boyd}: write the inequalities as
$-N_F(\xi)\le0$ and $-M_{H,\ell}(\xi)\preceq0$, with multipliers
$z_F\ge0$ and $Y_{H,\ell}\succeq0$, where $Y_{H,\ell}$ has the same
size as $M_{H,\ell}$. With the trace inner product
$\langle Y,M\rangle=\tr(YM)$, the Lagrangian is
\begin{align}
 L(\xi;z,Y)&=-\sum_Fz_FN_F(\xi)-
 \sum_{H,\ell}\langle Y_{H,\ell},M_{H,\ell}(\xi)\rangle\nonumber\\
 &=-c_0(z,Y)-\sum_{j=1}^q\xi_jc_j(z,Y),\label{eq:lagrangian}\\
 c_j(z,Y)&:=\sum_Fz_Fa_{Fj}+
 \sum_{H,\ell}\langle Y_{H,\ell},M_{H,\ell}^{(j)}\rangle.
 \nonumber
\end{align}
The definition of $c_j$ includes $j=0$.
The equalities are already imposed by \eqref{eq:affine}; taking the
infimum over $\R^q$ gives
\[
 g(z,Y)=\inf_\xi L(\xi;z,Y)=
 \begin{cases}-c_0(z,Y),&c_j(z,Y)=0\ (1\le j\le q),\\
 -\infty,&\text{otherwise}.
 \end{cases}
\]
Consequently the dual is
\begin{equation}\label{eq:dual}
 \begin{array}{ll}
 \text{maximize}_{z,Y}&-c_0(z,Y)\\
 \text{subject to}&c_j(z,Y)=0\quad(1\le j\le q),\\
 &z_F\ge0,\quad Y_{H,\ell}\succeq0.
 \end{array}
\end{equation}
For a primal feasible $\xi$, positive semidefiniteness gives
$g(z,Y)\le L(\xi;z,Y)\le0$. A positive dual value therefore excludes
a graph and, by scaling, can be normalized to $1$. We seek
\begin{equation}\label{eq:normaldual}
 c_0(z,Y)=-1,\qquad c_j(z,Y)=0\ (1\le j\le q),\qquad
 z_F\ge0,\quad Y_{H,\ell}\succeq0.
\end{equation}
This sufficient certificate uses only weak duality; no strict feasibility
or converse assertion about dual attainment is assumed.

\paragraph{From dual matrices to scalar identities.}
Index the rank-one terms of all positive semidefinite dual blocks by
a single finite set of indices $i$, recording the reference graph
$H_i$ and degree $\ell_i$ of each term. Then
\[
 Y_{H,\ell}=\sum_{i:\,(H_i,\ell_i)=(H,\ell)}\eta_i f_if_i^T,
 \qquad f_i\in\R^{2^{|H_i|}},\quad\eta_i>0.
\]
Thus \eqref{eq:normaldual} becomes
\begin{equation}\label{eq:certificate}
 \boxed{\displaystyle
 \sum_i\eta_iS(H_i,\ell_i,f_i;N(\xi))+
 \sum_Fz_FN_F(\xi)=-1\qquad(\xi\in\R^q).}
\end{equation}
For actual counts, every term on the left is nonnegative. The following
sufficient criterion isolates the mathematical conclusion from the search.

\begin{theorem}[Exact nonexistence certificate]\label{thm:exact-certificate}
Fix a parameter set and the counting and kernel constraints above.
For finitely many admissible pairs $(H_i,\ell_i)$ and rational
directions $f_i$, put $S_i(N)=S(H_i,\ell_i,f_i;N)$.
Suppose there are rational multipliers $\eta_i,z_F\ge0$ and a rational
number $\varepsilon>0$ such that
\begin{equation}\label{eq:exact-certificate-theorem}
 \sum_i\eta_iS_i(N)+\sum_Fz_FN_F=-\varepsilon
 \qquad\text{whenever }AN=b.
\end{equation}
Then no strongly regular graph with the prescribed parameters exists.

The same conclusion holds if proved linear equalities are appended to
$AN=b$ and the identity includes additional terms
$\theta_a\mathcal I_a(N)$ on the left, where $\theta_a\in\Q_{\ge0}$
and each rational affine expression $\mathcal I_a(N)$ is proved
nonnegative for the counts of every hypothetical graph with these
parameters. In that case the identity is required on the augmented
affine solution space.
\end{theorem}
\begin{proof}
The counts of a hypothetical graph satisfy all imposed equalities.
Each $S_i(N)$ is nonnegative by Proposition~\ref{prop:countingkernel},
each $N_F$ is nonnegative, and each additional $\mathcal I_a(N)$ is
nonnegative by hypothesis. The identity would give
$0\le-\varepsilon<0$.
\end{proof}
Count upper bounds, written as $U-x\ge0$, are examples of the
additional expressions. Section~\ref{sec:structural} also supplies
linearized integer inequalities and additional counting equalities.

\paragraph{Recovering rational weights.}
Choose rational directions $f_i$ and compute
\[
 \kappa_{ij}=f_i^TM_{H_i,\ell_i}^{(j)}f_i,\qquad
 S(H_i,\ell_i,f_i;N(\xi))=\kappa_{i0}+\sum_{j=1}^q\kappa_{ij}\xi_j.
\]
Then solve the rational linear feasibility problem
\begin{equation}\label{eq:rationallp}
 \eta_i,z_F\ge0,\qquad
 \sum_i\eta_i\kappa_{ij}+\sum_Fz_Fa_{Fj}
 =\begin{cases}-1,&j=0,\\0,&1\le j\le q.\end{cases}
\end{equation}
A real solution of this rational linear system implies a rational one.
Indeed, fix the coordinates that vanish at a real feasible point.
The remaining rational equalities define a nonempty rational affine
space by row reduction. Its rational points are dense, and sufficiently
close rational points preserve positivity of the remaining coordinates.
This implication assumes feasibility after the rational directions
have been fixed: neither a numerical dual solution nor its rational
approximation guarantees that feasibility. We retain only multipliers
satisfying every equality and sign condition exactly; each block
$\eta_if_if_i^T$ is then positive semidefinite by construction.

Verify the parametrization by
\begin{equation}\label{eq:affineverify}
 A(a_{F0})_F=b,\qquad A(a_{Fj})_F=0\ (1\le j\le q),\qquad
 \rank(a_{Fj})_{F,1\le j\le q}=q=|\mathcal G_{\le p}|-\rank A.
\end{equation}
Together with \eqref{eq:c} and \eqref{eq:rationallp}, these checks verify
the hypotheses of Theorem~\ref{thm:exact-certificate}, independently
of the numerical search.

\section{Exact contradictions from the counting and kernel constraints}\label{sec:direct}
The spherical data are determined by \eqref{eq:spectral}--\eqref{eq:sphere}.
We give seven exact contradictions. The displayed functions and the
coefficients in the appendices and Supplement~\cite{supp405} specify
each certificate completely.

\subsection{Four five-point certificates}
Use $p=5$, $m=3$, and degrees $\ell\in\{1,2\}$.  The $52$ graph types
on at most five vertices give a counting system of rank $48$ for each
of $v=540,550,703,1344$.  Its four non-pivot columns, in lexicographic
edge-string order, are
\begin{equation}\label{eq:direct-free-types}
 F_1=K_5\setminus P_3,\quad F_2=K_5\setminus 2K_2,\quad
 F_3=K_5\setminus K_2,\quad F_4=K_5,
 \qquad \xi_j=N_{F_j}.
\end{equation}
Here $K_5\setminus P_3$ means deletion of the two edges of a
three-vertex path.  These are the original labelled counts, so
$\xi_j\ge0$.  Their automorphism orders are $4,8,12,120$, respectively.
The corresponding ten-bit edge strings are
$0011111111$, $0111111011$, $0111111111$, and $1111111111$.
They use the edge order specified in Appendix~\ref{app:verification}.
The total-count equations in Section~\ref{sec:method} may be retained: they already
follow from the empty-set extension equations and $N_{K_1}=v$.

For three reference vertices, an edge string lists $01,02,12$.
Thus $001$ is one edge and an isolated vertex, $011$ is a path, and
$111$ is a triangle; their automorphism orders are $2,2,6$.
The functions in Table~\ref{tab:direct-functions} are given by all eight
values in exactly the order of Section~\ref{sec:method}:
\[
 \varnothing,\{2\},\{1\},\{1,2\},\{0\},\{0,2\},\{0,1\},\{0,1,2\}.
\]
Index the rows within each $v$-block of Table~\ref{tab:direct-functions}
starting at $1$. A row displaying $qf_i$ defines $f_i$ by dividing
every entry by the row's denominator $q$.
All reference Gram matrices are positive definite.  Indeed, in the
four cases $2\max(|\alpha|,|\beta|)<1$, so strict diagonal dominance
proves this for every three-vertex graph.

\begin{table}[htbp]\centering\small
\caption{Explicit rational functions for the four five-point proofs.
The last column supplies $t_i>0$ in \eqref{eq:direct-contradiction}.}
\label{tab:direct-functions}
\setlength{\tabcolsep}{4pt}
\begin{tabular}{rrrrlr}\toprule
$v$&$H_i$&$\ell_i$&$q$&$qf_i$&$t_i$\\\midrule
540&001&1&10&$(10,6,6,0,7,0,0,0)$&102\\
&111&2&10&$(3,-4,-4,10,-4,10,10,0)$&80\\
&011&1&10&$(10,4,9,0,9,0,0,0)$&61\\
&011&1&10&$(10,5,8,0,8,0,0,0)$&37\\
&011&1&10&$(10,4,8,0,8,0,0,0)$&22\\
\midrule
550&001&1&10&$(1,2,2,-5,4,-7,-7,10)$&11\\
&011&1&10&$(1,1,-5,4,-5,4,10,-8)$&22\\
&011&2&10&$(2,-10,-2,3,-2,3,5,1)$&42\\
&111&1&10&$(-3,5,5,-8,5,-8,-8,10)$&1\\
&111&2&10&$(2,-10,-10,9,-10,9,9,1)$&105\\
\midrule
703&001&1&10&$(1,2,2,-5,4,-7,-7,10)$&1\\
&001&2&10&$(-1,-4,-4,-5,-6,6,6,10)$&4\\
&011&2&10&$(1,-10,-2,5,-2,5,5,-2)$&14\\
&111&2&10&$(2,-10,-10,9,-10,9,9,0)$&2\\
\midrule
1344&001&1&30&$(30,17,17,0,24,0,0,0)$&11\\
&011&2&30&$(21,7,20,30,20,30,30,0)$&9\\
&011&1&30&$(20,-7,30,0,30,0,0,0)$&2\\
&011&1&30&$(21,-6,30,0,30,0,0,0)$&4\\
&001&1&30&$(30,17,17,0,25,0,0,0)$&1\\
\bottomrule\end{tabular}\end{table}

Substitute each displayed $f_i$ into the scalar coefficient formula of
Section~\ref{sec:method}, and then substitute the exact counting parametrization
$N=N(\xi)$.  Rational arithmetic gives
\begin{equation}\label{eq:direct-contradiction}
 \sum_i\frac{t_i}{|\operatorname{Aut}H_i|}
 S(H_i,\ell_i,f_i;N(\xi))
 =-b_0-\sum_{j=1}^4b_j\xi_j,
\end{equation}
where all five $b_j$ are strictly positive; their values are in
Appendix~\ref{app:five-coefficients}, Tables~\ref{tab:direct-coeff-a} and~\ref{tab:direct-coeff-b}.
Every summand on the left is nonnegative for actual graph counts,
whereas the right side is strictly negative.

Moving the count terms to the left and dividing by $b_0$ gives
the certificate of Theorem~\ref{thm:exact-certificate} with
$\varepsilon=1$, $\eta_i=t_i/(b_0|\Aut H_i|)$,
$z_{F_j}=b_j/b_0$ for $1\le j\le4$, and all other $z_F=0$.

\paragraph{The contradiction for each parameter set.}
For each $v$-block, abbreviate
\[
 E_{v,i}=\frac{S(H_i,\ell_i,f_i;N)}{|\operatorname{Aut}H_i|}\ge0.
\]
The positivity holds for every hypothetical graph with the indicated
parameters. Write $b_{v,j}$ for the corresponding entries of
Tables~\ref{tab:direct-coeff-a}--\ref{tab:direct-coeff-b}.
The four instances of \eqref{eq:direct-contradiction} now give the
following separate numerical contradictions, using only $\xi_j\ge0$.
For $(540,147,66,30)$,
\begin{align}
0&\le 102E_{540,1}+80E_{540,2}+61E_{540,3}
       +37E_{540,4}+22E_{540,5}\nonumber\\
 &= -b_{540,0}-\sum_{j=1}^4b_{540,j}\xi_j
 \le-\frac{574630684892568}{6496825}<0.\label{eq:contradiction-540}
\end{align}
For $(550,162,75,36)$,
\begin{align}
0&\le11E_{550,1}+22E_{550,2}+42E_{550,3}
       +E_{550,4}+105E_{550,5}\nonumber\\
 &\le-\frac{3153820063124750}{4826151}<0.\label{eq:contradiction-550}
\end{align}
For $(703,182,81,35)$,
\begin{align}
0&\le E_{703,1}+4E_{703,2}+14E_{703,3}+2E_{703,4}\nonumber\\
 &\le-\frac{827108257923964177}{2718487200}<0.\label{eq:contradiction-703}
\end{align}
For $(1344,221,88,26)$,
\begin{align}
0&\le11E_{1344,1}+9E_{1344,2}+2E_{1344,3}
       +4E_{1344,4}+E_{1344,5}\nonumber\\
 &\le-\frac{11861558480040192}{18036575}<0.\label{eq:contradiction-1344}
\end{align}
These four certificates use only degrees $1$ and $2$.
After the counting substitutions, the displayed positive combination
has a strictly negative constant term and nonpositive coefficients
on the four remaining nonnegative counts.

\subsection{The six-point cases 351 and 476}
For $(351,140,73,44)$ and $(476,133,60,28)$ we retain counts through
$p=6$ and the degree-one and degree-two kernels on three and four
reference vertices. In these two cases, respectively,
$3\max(|\alpha|,|\beta|)=24/35$ and $15/19$, both less than $1$.
Thus every reference Gram matrix of order at most four is positive
definite by strict diagonal dominance. This is the finite combination of the constraints
of Section~\ref{sec:method} for $m=3,4$.  There are $208$ graph types, counting rank
$185$, and $23$ free counts.  No extra local exclusion or integrality
condition is used.

Let $S_{v,i}=S(H_i,\ell_i,f_i)$, where the rational functions are
listed in Appendix~\ref{app:directions}. Appendix~\ref{app:weights}
specifies positive integers $L_v,\omega_{v,i},\zeta_{v,j}$ and the
count graphs $G_{v,j}$. Put
\[
 \eta_{v,i}=\frac{\omega_{v,i}}{vL_v},\qquad
 z_{v,j}=\frac{\zeta_{v,j}}{vL_v}.
\]
Expanding \eqref{eq:c} and applying the counting equations gives,
for $(351,140,73,44)$,
\begin{equation}\label{eq:six351-explicit}
 0\le\sum_{i=1}^{17}\eta_{351,i}S_{351,i}
       +\sum_{j=1}^{7}z_{351,j}N_{G_{351,j}}=-1<0,
\end{equation}
and, for $(476,133,60,28)$,
\begin{equation}\label{eq:six476-explicit}
 0\le\sum_{i=1}^{16}\eta_{476,i}S_{476,i}
       +\sum_{j=1}^{8}z_{476,j}N_{G_{476,j}}=-1<0.
\end{equation}
These are two instances of Theorem~\ref{thm:exact-certificate}
with $\varepsilon=1$.

\subsection{The seven-point certificate for (405,132,63,33)}
\label{subsec:405-counting}
For these parameters the restricted eigenvalues are $33,-3$,
the embedding dimension is $d=30$, and the full spherical Gram matrix is
\[
 G=\frac{9I+3\operatorname{Adj}(\Gamma)-J}{8},\qquad
 \alpha=-\frac18,\quad\beta=\frac14.
\]
The constraints of Section~\ref{sec:method} exclude this case with
counting order $p=7$, at most five reference vertices, and degrees
$\ell=0,1,2$.  

\begin{proposition}\label{prop:405-counting}
There is no strongly regular graph with parameters $(405,132,63,33)$.
An exact certificate uses $67$ nonnegative count terms and $107$
weighted kernel sums: $54$ of degree zero, $50$ of degree one, and three
of degree two.
\end{proposition}
\begin{proof}[Exact certificate]
For an integer $h\ge0$, write $(v)_h=v(v-1)\cdots(v-h+1)$,
with $(v)_0=1$. Use the positive rescaling
\[
 y_F=\frac{2^{\binom h2}N_F}{(405)_h},\qquad h=|F|.
\]
All $1252$ graph types through order seven are retained, with
$y_{K_1}=1$.  The $4135$ ordinary one-vertex extension equations
admit $1078$ triangular elimination steps, leaving $173$ free
scaled graph counts and the constant coordinate. Write
$y=Px$, with $y\in\R^{1252}$,
$x=(1,x_1,\ldots,x_{173})^T$, and $P\in\Q^{1252\times174}$,
where $x_j=y_{C_j}$, $1\le j\le173$, for the free graph types.
The supplied rational matrix $P$ is checked against every
extension equation, and its free rows form an identity matrix.

For an anchor $H$ of order $m$, the matrix used in the stored
certificate is the positive multiple
\[
 \widehat M_{H,\ell}(y)=
 \frac{2^{\binom{m+1}{2}}}{(405)_{m+1}}M_{H,\ell}(N).
\]
This normalization follows directly by separating diagonal and distinct-vertex terms
in the definition of $S(H,\ell,f)$; the ratio of the
pair and single-vertex factors is $(404-m)/2^{m+1}$.
Each selected anchor Gram matrix is positive definite. The integer
directions are defined on all $2^m$ profiles; the three degree-two
energies use $H=K_{2,3}$.

Let $R_i(y)\le0$ be the $174$ selected sources, numbered as in
Supplement~\cite{supp405}: the first $67$ are $-y_F$, and the
remaining $107$ are $-f_i^T\widehat M_{H_i,\ell_i}(y)f_i$.
The kernel functions retain the source indices $i=68,\ldots,174$.
Every direction $f_i$ and every full positive rational coefficient
$w_i$ is printed there, together with formulas that reconstruct
all row coefficients using the ordinary counting equations.

For example, source $i=76$ has $H=K_{2,3}$ with bipartition
$\{0,1\}\cup\{2,3,4\}$ and degree $\ell=2$. Its $32$ values,
read row by row in binary lexicographic profile order, are
\[
 {\begin{array}{rrrrrrrr}
 117610&-293584&-293584&-14274&-293584&-14274&-14274&0\\
 -34929&-18491&-18491&265514&-18491&265514&265514&52070\\
 -34929&-18491&-18491&265514&-18491&265514&265514&52070\\
 0&-13820&-13820&-290273&-13820&-290273&-290273&201821
 \end{array}}
\]
and the exact coefficient of its source inequality is
\[
 w_{76}=\frac{386746425034437}{53675406395459206542968750000000}.
\]
This one source is part of the complete positive combination; its
positivity alone is not claimed to give the contradiction.

Applying the stated counting substitutions and collecting coefficients
in the full combination gives the exact identity
\[
 \sum_{i=1}^{174}w_iR_i(Px)=c_0+\sum_{j=1}^{173}c_jx_j.
\]
Here the $c_j$ are local to this combination of nonpositive sources.
The empty-set extension equations and count nonnegativity imply
$0\le y_F\le2^{\binom{|F|}{2}}$ by induction on $|F|$.
Put $U_j=2^{\binom{|C_j|}{2}}$ and
$E_{405}=-\sum_{c_j<0}c_jU_j$. Exact rational arithmetic yields the
following convenient bounds on the two sides of the residual payment:
\[
 c_0>\frac{570}{10^{10}},\qquad
 0\le E_{405}<\frac{11}{10^{10}}.
\]
In particular, adding explicit nonpositive count-bound terms produces
\begin{align}
 0\;&\ge\sum_{i=1}^{174}w_iR_i(Px)
       +\sum_{c_j\ge0}c_j(-x_j)
       +\sum_{c_j<0}(-c_j)(x_j-U_j)\notag\\
 &=c_0-E_{405}=\Delta_{405}
   >\frac{559}{10^{10}}>0.
 \label{eq:405-counting-margin}
\end{align}
After multiplying by $-1$, this is Theorem~\ref{thm:exact-certificate}
with $\varepsilon=\Delta_{405}$ and the proved count-bound inequalities.
Supplement~\cite{supp405} specifies all
$174$ summands and the counting recursion used to derive $c_0$ and
$c_j$; the supplementary verifier independently reconstructs all rows
and checks both strict rational bounds above.
\end{proof}

The five-point real-count relaxation is feasible at every Gegenbauer
degree; Supplement~\cite[Section~S2]{supp405} gives an exact witness and proof.

\section{Adding structural information for three further exclusions}
\label{sec:structural}

We now treat
\[
 (375,102,45,21),\qquad(441,88,35,13),\qquad(575,112,45,16).
\]
We add integer common-neighbour inequalities and unique-extension
identities to the Gegenbauer positivity and counting information of
Section~\ref{sec:method}. Both are expressed in the same labelled counts
$N_F$. These proved conditions strengthen the same counting system.

\subsection{Integer common-neighbour counts}
\label{subsec:integer-triples}
Let $H$ be a labelled graph on three vertices.  For an induced embedding
$\phi:H\hookrightarrow\Gamma$, let $t_\phi$ be the number of vertices
outside $\phi([3])$ adjacent to all three selected vertices.  Write
$(t)_j=t(t-1)\cdots(t-j+1)$ and $(t)_0=1$, and define
\[
 T_j(H)=\sum_{\phi:H\hookrightarrow\Gamma}(t_\phi)_j.
\]
Let $\mathcal E_j(H)$ be the family of labelled graphs obtained from $H$
by adding $j$ new labelled vertices, each adjacent to all of $H$, with
every possible edge pattern among the new vertices.  Counting ordered
distinct common neighbours gives
\begin{equation}\label{eq:struct-moment}
 T_j(H)=\sum_{F\in\mathcal E_j(H)}N_F.
\end{equation}
Here $T_0(H)=N_H$; the labels record the ordered neighbours, so no
factorial or automorphism factor occurs. Isomorphic extensions retain
their full multiplicity.

For integers $a,b$, the polynomial
\begin{equation}\label{eq:struct-quartic}
 p_{a,b}(t)=(t-a)(t-a-1)(t-b)(t-b-1)
\end{equation}
is nonnegative at every integer $t$.  Thus
$\sum_\phi p_{a,b}(t_\phi)\ge0$.  Expanding in the basis
$(t)_0,\ldots,(t)_4$ and using~\eqref{eq:struct-moment} gives a
linear inequality in counts of graphs on at most seven vertices.
For example,
\[
 p_{1,4}(t)=(t)_4-6(t)_3+20(t)_2-40(t)_1+40,
\]
so the $375$ case uses
\[
 T_4(\overline{K_3})-6T_3(\overline{K_3})
 +20T_2(\overline{K_3})-40T_1(\overline{K_3})
 +40T_0(\overline{K_3})\ge0.
\]
The complete list of quartics with positive multipliers in the final
certificates is
\begin{center}
\begin{tabular}{rrrr}
\toprule
$v$&$H$&$a$&$b$\\
\midrule
375&$\overline{K_3}$&1&4\\
441&$\overline{K_3}$&0&4\\
441&$K_3$&18&28\\
575&$\overline{K_3}$&1&6\\
575&$\overline{K_3}$&1&7\\
575&$P_3$&8&14\\
\bottomrule
\end{tabular}
\end{center}
Each inequality holds at every reference copy before averaging.

\subsection{An integral Gram matrix and its short vectors}
\label{subsec:struct-short}
The restricted eigenvalues are $r$ and $-3$, with $r$ of multiplicity $d$.
Put
\[
 c=r+3,\qquad \rho=\frac{k+3}{v},\qquad
 B=\operatorname{Adj}(\Gamma)+3I.
\]
The standard spherical Gram matrix satisfies
$(3-\rho)G=B-\rho J$.  Choose a unit vector $e_0$ perpendicular
to the spherical span and set
\[
 b_x=\sqrt{3-\rho}\,p_x+\sqrt\rho\,e_0,
 \qquad \delta=\sqrt\rho\,e_0.
\]
Then $B$ is the Gram matrix of the $b_x$, and
\begin{equation}\label{eq:struct-integral}
 \langle b_x,b_y\rangle=
 \begin{cases}3&x=y,\\1&x\sim y,\\0&x\ne y,\ x\not\sim y.
 \end{cases}
\end{equation}
Moreover $\delta=v^{-1}\sum_xb_x$ and
$\langle b_x,\delta\rangle=\|\delta\|^2=\rho$.
The SRG identities give
\begin{equation}\label{eq:struct-B}
 B^2=cB+\mu J,\qquad
 K:=B-\rho J,\qquad K^2=cK,\qquad
 \operatorname{rank}K=d.
\end{equation}

Consider the integer span $\Lambda=\sum_x\mathbb Zb_x$.
For $q=\sum_xh_xb_x$ with $h_x\in\mathbb Z$, define
\[
 \tau=\sum_xh_x,\qquad Q=\|q\|^2,\qquad z_x=\langle q,b_x\rangle.
\]
Since $\langle q,\delta\rangle=\rho\tau$, the value of $\tau$ is independent
of the representation of $q$. Equation~\eqref{eq:struct-B} gives
\begin{equation}\label{eq:struct-moments}
 \sum_xz_x=(k+3)\tau,\qquad
 \sum_xz_x^2=cQ+\mu\tau^2,\qquad
 \operatorname{Adj}(\Gamma)z=rz+\mu\tau\mathbf1.
\end{equation}
Also
\begin{equation}\label{eq:struct-parity}
 Q\ge\rho\tau^2,\qquad Q\equiv\tau\pmod2,
\end{equation}
where the first assertion is Cauchy's inequality with $\delta$ and
the second follows by expanding the integral Gram form modulo two.
Since every $z_x$ is an integer,
\begin{equation}\label{eq:struct-gap}
 0\le\sum_xz_x(z_x-1)=cQ+\mu\tau^2-(k+3)\tau.
\end{equation}

\begin{lemma}\label{lem:struct-shell}
For each of the three parameter sets in this section, $\Lambda$
contains no vector of squared norm one or two.  Every vector of
squared norm three and coefficient sum one equals $b_x$ for a unique
vertex $x$.
\end{lemma}
\begin{proof}
All three sets satisfy $1/8<\rho<1/2$ and $\mu<r-2$.
If $Q=1$,~\eqref{eq:struct-parity} forces $\tau=\pm1$.
Changing sign gives $\tau=1$, for which the right side of
\eqref{eq:struct-gap} is $-2r<0$.
If $Q=2$, only $\tau=0,\pm2$ are possible.  For $\tau=2$ that
right side is $2\mu-4r<0$, and sign change handles $\tau=-2$.

Suppose that $Q=2$ and $\tau=0$.  Cauchy's inequality gives
$|z_x|\le\sqrt6$.  If $z_x=2$ or $-2$, respectively $q-b_x$
or $q+b_x$ would have squared norm one.  Hence
$z_x\in\{-1,0,1\}$.  By~\eqref{eq:struct-moments}, the sets
$P=\{x:z_x=1\}$ and $N=\{x:z_x=-1\}$ both have size $c$.
For $p\in P$, the equation $\operatorname{Adj}(\Gamma)z=rz$
gives
\begin{equation}\label{eq:struct-defect}
 |\{x\in P\setminus\{p\}:x\not\sim p\}|
 +|\mathcal N_\Gamma(p)\cap N|=2.
\end{equation}
The analogous equation holds in $N$.  If two vertices of $P$ were
nonadjacent, each would have at least $r$ neighbours in $P$; these
would give at least $2r-(c-2)=r-1>\mu$ common neighbours.
Thus $P$ and $N$ are cliques, and the cross-edges form a bipartite
$2$-regular graph $C$, a disjoint union of even cycles.

Let $s$ be the all-one column on $P\cup N$, let $e$ be $1$ on $P$
and $-1$ on $N$, and set $R=\operatorname{diag}(e)$.  Then
\begin{equation}\label{eq:struct-root-rank}
 K_{P\cup N}=R(2I-\operatorname{Adj}(C))R
 +(1/2-\rho)ss^T+\tfrac12ee^T.
\end{equation}
If $C$ has $m$ cycles, the first term has rank $2c-m$.
The vector $e$ is a nonzero kernel vector of that term, so its
rank rises by one after adding $ee^T/2$; the other term is PSD.
Since every cycle has length at least four,
\[
 \operatorname{rank}K_{P\cup N}
 \ge2c-\lfloor c/2\rfloor+1.
\]
This is $46>34$ for $375$ and $54>46$ for $575$, contradicting
\eqref{eq:struct-B}.  For $441$, the remaining rank-boundary
cases are excluded in the next subsection, using only identities
already established above.

Once squared norms one and two are excluded, let $Q=3$, $\tau=1$.
Now $|z_x|\le3$.  A coordinate $3$ gives $q=b_x$; a coordinate
$-3$ would give $q=-b_x$, incompatible with $\tau=1$.
A coordinate $2$ or $-2$ would give a squared-norm-two vector.
If none of these occur, $z_x\in\{-1,0,1\}$ and
\[
 |\{x:z_x=-1\}|=
 \frac{3c+\mu-(k+3)}2=3.
\]
At a negative coordinate, the last identity in
\eqref{eq:struct-moments} says that the sum of neighbouring
coordinates is $\mu-r<-2$.  But only two other negative
coordinates are available, a contradiction.  Finally distinct
vertices have distinct vectors by~\eqref{eq:struct-integral},
which proves uniqueness.
\end{proof}

\subsection{The rank-boundary argument for 441}
\label{subsec:441-boundary}
Here $c=28$, $\rho=13/63$, and $d=44$.
For the centered vectors $u_x=b_x-\delta=\sqrt{3-\rho}\,p_x$, the nonzero eigenvalues
of their Gram matrix $K$ are all $28$, so
\begin{equation}\label{eq:struct-tight}
 \sum_xu_xu_x^T=28I_{44}.
\end{equation}
Suppose that the unresolved vector with $Q=2$, $\tau=0$ exists.
The preceding argument gives cliques $P,N$ of size $28$ with
cross-edges a union of $m$ even cycles of lengths
$2d_1,\ldots,2d_m$, where $d_i\ge2$ and $\sum_i d_i=28$.
On these $56$ vertices put
\[
 B_0=2I+\operatorname{Adj}(C),\qquad
 s=\mathbf1_P+\mathbf1_N,\qquad e=\mathbf1_P-\mathbf1_N,
 \qquad \beta_0=\tfrac12-\rho.
\]
For $H=\Gamma[P\cup N]$, their centered Gram matrix is
$K_H=B_0+\beta_0ss^T+ee^T/2$.
The matrix $B_0$ is PSD, with one alternating kernel vector per
cycle.  The vector $s$ belongs to its range and $e$ is a nonzero
kernel vector.  Hence
\[
 \operatorname{rank}K_H=56-m+1\le44,
\]
so only $m=13$ and $m=14$ remain.

Since $\tau=0$, the vector $q$ belongs to the centered spherical span.
For an external vertex $x$, let $a\in\{0,1\}^{56}$ record its
neighbours in $P\cup N$.  Its neighbour totals in $P$ and $N$
coincide.  Indeed, applying~\eqref{eq:struct-tight} to the putative
vector $q$ gives $\sum_Pu_y-\sum_Nu_y=28q$, whose inner product
with $u_x$ is zero. The singular Gram-matrix range condition
\cite[Appendix~A.5.5]{boyd} requires $a-\rho s$ to be orthogonal to
$\ker K_H$. This kernel consists of combinations of the
alternating cycle vectors orthogonal to $e$.  Consequently the
neighbour difference between the two sides of cycle $i$, divided
by $d_i$, is independent of $i$.  Its total is zero, so $a$ is
balanced on each cycle and lies in the range of $B_0$.

Set
\[
 T=|\mathcal N_\Gamma(x)\cap P|=|\mathcal N_\Gamma(x)\cap N|,
 \qquad E=a^TB_0^+a,
\]
where $B_0^+$ is the inverse on its range and zero on its kernel.
The positive eigenvalues of $B_0$ are at most $4$, giving
$E\ge\|a\|^2/4=T/2$.
Since $B_0^+s=s/4$, a rank-one inverse calculation gives the
squared residual length after projecting $u_x$ away from the
anchor span:
\begin{align}
 \gamma_x
 &=3-\rho-E+\rho T-\frac{\rho^2c}{2}
   +\frac{\beta_0(T-c\rho)^2}{4(1+\beta_0c/2)}\notag\\
 &=\frac{37T^2+104T+6896}{2576}-E.
 \label{eq:struct-441-height}
\end{align}
The $ee^T/2$ term contributes nothing because the profile is balanced.

If $m=13$, the cycles are twelve $C_4$ and one $C_8$, or eleven
$C_4$ and two $C_6$.  Their span has dimension $44$, so
$\gamma_x=0$.  The complete possible energies for a balanced
Boolean profile on a single cycle are
\[
\begin{array}{c|c|l}
\text{cycle}&\text{profiles}&\text{energies}\\\hline
C_4&6&0,\ 3/4,\ 1\\
C_6&20&0,\ 5/6,\ 4/3,\ 3/2,\ 2\\
C_8&70&0,\ 7/8,\ 3/2,\ 15/8,\ 2,\ 23/8,\ 7/2
\end{array}
\]
These rational values follow by inverting $B_0+ww^T$ on the
single cycle, where $w$ is its alternating column, and evaluating
all balanced Boolean columns.  Thus the total $E$ belongs to
$\frac18\mathbb Z$ or $\frac1{12}\mathbb Z$.
As $2576=16\cdot7\cdot23$,~\eqref{eq:struct-441-height} forces
$37T^2+104T+6896$ to vanish modulo $7$ and $23$.  But
\[
 37T^2+104T+6896\equiv2(T-2)^2\pmod7,
 \qquad
 37T^2+104T+6896\equiv14(T-16)^2\pmod{23}.
\]
For $0\le T\le28$ this forces $T=16$, and hence $E=7<T/2$,
a contradiction.

If $m=14$, all cycles are $C_4$ and their span has dimension $43$.
Let $n_1$ count cycles where the profile selects one vertex from
each side, and $n_2$ those where it selects all four vertices.
Then
\[
 n_1,n_2\ge0,\quad n_1+n_2\le14,\qquad
 T=n_1+2n_2,\qquad E=\tfrac34n_1+n_2.
\]
Substitution in~\eqref{eq:struct-441-height}, over these $120$
integer pairs, gives exactly the following possibilities with
$\gamma_x\ge0$:
\[
\begin{array}{c|c}
n_2& n_1\\\hline
0&0,1,2,3,4\\
1&0,1,2\\
2&0,1\\
3&0\\
13&0,1\\
14&0
\end{array}
\]
In each of these $14$ cases direct rational substitution gives
\begin{equation}\label{eq:struct-441-lower}
 \gamma_x\ge\frac{11}{161}+\frac{237}{8372}n_2(x).
\end{equation}
On each $C_4$, the two vertices in $P$ have $28$ common neighbours
inside $P\cup N$, hence $35-28=7$ outside.  The balance condition
on that cycle makes these precisely its contributions to $n_2$.
Consequently $\sum_{x\notin P\cup N}n_2(x)=14\cdot7=98$.
Summing~\eqref{eq:struct-441-lower} over the $385$ external vertices
yields
\[
 \sum_{x\notin P\cup N}\gamma_x
 \ge385\frac{11}{161}+98\frac{237}{8372}
 =\frac{17389}{598}>28.
\]
Projection of~\eqref{eq:struct-tight} onto the one-dimensional
orthogonal complement of the anchor span instead gives
$\sum_x\gamma_x=28$.  The anchor residuals vanish, so this is the
required contradiction and completes the boundary step in
Lemma~\ref{lem:struct-shell}.

\subsection{Unique extensions and new equations in the counts}
\label{subsec:struct-completions}
For a labelled graph $H$, put $B_H=3I+\operatorname{Adj}(H)$.
By Lemma~\ref{lem:struct-shell}, an integer column $h$ with
\begin{equation}\label{eq:struct-completion}
 \mathbf1^Th=1,\qquad h^TB_Hh=3,\qquad
 B_Hh\in\{0,1\}^{|H|}
\end{equation}
forces a unique external vertex at every induced embedding of $H$.
Indeed, the vector sum is a vertex vector; its inner products with
the selected vertices give exactly the indicated adjacency profile,
and none is three, so it is external.  A second vertex with that
profile has inner product $h^TB_Hh=3$ with the same vector sum
and must coincide with it.  If $F$ is this labelled one-vertex
extension, restriction and completion are inverse bijections, so
\begin{equation}\label{eq:struct-bijection}
 N_H=N_F.
\end{equation}
Likewise an integer column with $h^TB_Hh\in\{1,2\}$ proves
$N_H=0$.

For example, suppose $x_1,x_2$ and $y_1,y_2,y_3$ induce
$K_{2,3}$.  The vector
\[
 b_{y_1}+b_{y_2}+b_{y_3}-b_{x_1}-b_{x_2}
\]
has coefficient sum one and squared norm $15-12=3$.
It is the unique missing vertex completing this copy to $K_{3,3}$.
With the labelled-count convention of Section~2, the resulting
identity is simply $N_{K_{2,3}}=N_{K_{3,3}}$.

One can obtain an equation on seven-point counts even when the
unique extension has eight vertices.  Suppose an eight-vertex
graph $F$ has an integer column $a$ satisfying
\[
 B_Fa=0,\qquad \mathbf1^Ta=0.
\]
For an index $i$ with $a_i=\pm1$, put $H_i=F-i$ and
$h^{(i)}=-a_i(a_j)_{j\ne i}$.  The displayed null relation shows
that $h^{(i)}$ satisfies~\eqref{eq:struct-completion}, with
profile equal to that of the deleted vertex.  Hence, if also
$a_j=\pm1$,
\begin{equation}\label{eq:struct-cards}
 N_{H_i}=N_F=N_{H_j}.
\end{equation}
Eliminating $N_F$ leaves an equation involving only seven-vertex counts.

\subsection{The three exact contradiction certificates}
\label{subsec:struct-certificates}
The source inequalities are count nonnegativity, the degree $0,1,2$
kernel sums $S(H,\ell,f)$, and the six quartic inequalities above.
The zero-count and completion equations follow from
Lemma~\ref{lem:struct-shell} and~\eqref{eq:struct-completion}.
The rational source lists are identified in Appendix~\ref{app:delivery}; their conversion
to the notation of Section~\ref{sec:method} is given in
Appendix~\ref{app:structural-normalization}.

Use the positive scaling
\begin{equation}\label{eq:struct-scaled}
 y_F=\frac{2^{\binom h2}N_F}{(v)_h},\qquad h=|F|,
 \qquad (v)_h=v(v-1)\cdots(v-h+1).
\end{equation}
Thus $y_{K_1}=1$. Ordinary one-vertex equations eliminate a vertex
of degree at most two by
\[
 N_{H_B}=e(H,B)N_H-\sum_{B\subsetneq B'\subseteq V(H)}N_{H_{B'}},
 \qquad |B|\le2.
\]
Each substitution decreases the number of vertices or increases the
number of edges at fixed order. Through order seven, it leaves the
singleton and $173$ graph types of minimum degree at least three.
The proved zero-count and completion equations then have rank $63$,
leaving $110$ free scaled graph counts $x_j=y_{C_j}$ and the constant $1$.
All substitutions are rational consequences of the displayed equations.

Every $C_j$ is connected: two components of minimum degree three would
require at least eight vertices. Choosing a spanning tree gives
\[
 N_{C_j}\le vk^{h-1},\qquad
 0\le x_j\le
 \frac{2^{\binom h2}k^{h-1}}{(v-1)_{h-1}}<U_v,
 \qquad h=|C_j|.
\]
For $4\le h\le7$, direct substitution permits
$U_{375}=1024$ and $U_{441}=U_{575}=256$.

After these substitutions, write the selected source inequalities as
$R_i(x)\le0$. Their positive rational combination is
\[
 R_v(x)=\sum_iw_iR_i(x)
       =c_{v,0}+\sum_{j=1}^{110}c_{v,j}x_j,
 \qquad w_i\in\mathbb Q_{>0}.
\]
Adding the indicated nonpositive count-bound terms cancels every
remaining variable:
\begin{align}
 \mathcal L_v(x):={}&R_v(x)
   +\sum_{c_{v,j}\ge0}c_{v,j}(-x_j)
   +\sum_{c_{v,j}<0}(-c_{v,j})(x_j-U_v)\notag\\
 ={}&c_{v,0}+U_v\sum_{c_{v,j}<0}c_{v,j}=:\Delta_v.
 \label{eq:struct-margin}
\end{align}
Every term on the first line is nonpositive for a hypothetical graph.
The exact source combinations yield the three contradictions
\begin{equation}
 \begin{aligned}
 0&\ge\mathcal L_{375}(x)=\Delta_{375}>\frac7{10^9}>0,\\
 0&\ge\mathcal L_{441}(x)=\Delta_{441}>\frac9{10^7}>0,\\
 0&\ge\mathcal L_{575}(x)=\Delta_{575}>\frac3{10^6}>0.
 \end{aligned}
 \label{eq:struct-three-contradictions}
\end{equation}
Negating \eqref{eq:struct-margin} gives the extended certificate in
Theorem~\ref{thm:exact-certificate}, with $\varepsilon=\Delta_v$.
Thus none of the three parameter sets in this section is realizable.
The source inventories are recorded in Appendix~\ref{app:structural-normalization}.
Only finite rational sums and comparisons are needed to evaluate these
certificates; the search that produced them is not part of the proof.

\vspace{\baselineskip}
\appendix
\section{Exact evaluation and certificate data}\label{app:verification}
\subsection{Polynomial evaluation and singular references}
We use the normalized Gegenbauer recurrence in
\cite[Section~2.1]{musin}:
\[
 P_0^{(D)}(t)=1,\qquad P_1^{(D)}(t)=t,\qquad
 P_\ell^{(D)}(t)=
 \frac{(2\ell+D-4)tP_{\ell-1}^{(D)}(t)
 -(\ell-1)P_{\ell-2}^{(D)}(t)}{\ell+D-3},\quad\ell\ge2.
\]
For $D=2$ this is the Chebyshev recurrence. At residual inner product
$z$ and squared residual lengths $\eta,\eta'$, the homogeneous kernels
satisfy
\[
 K_0=1,\qquad K_1=z,\qquad
 K_\ell=
 \frac{(2\ell+D-4)zK_{\ell-1}
       -(\ell-1)\eta\eta'K_{\ell-2}}{\ell+D-3},\quad\ell\ge2.
\]
Their coefficients are rational when the spherical data and the inverse
reference Gram matrix are rational. The diagonal value is $1$ at degree
zero and $\eta^\ell$ at positive degree, including $\eta=0$.

Every anchor used for the kernels in Section~\ref{sec:direct} has a
positive definite Gram matrix. For a singular anchor, projection would
instead use its actual span, residual dimension $d-\operatorname{rank}G_H$,
and the corresponding range conditions. The singular configurations in
Section~\ref{sec:structural} are treated explicitly by their separate
structural arguments.

\subsection{Certificate snapshot and verification}\label{app:delivery}
The public companion repository is
\begin{center}\small
\url{https://github.com/SelenaJXWANG/srg-gegenbauer-counting}
\end{center}
For reproducibility, the certificate snapshot dated \texttt{2026-10-14}
is fixed at commit
\href{https://github.com/SelenaJXWANG/srg-gegenbauer-counting/tree/85cd39b74c9ef8a61ab746deb1102af3eeae7e10}{\texttt{85cd39b}}~\cite{certificateData}.
Readers may download that snapshot using GitHub's source archive or
check it out with Git. This manuscript revision adds repository
references without altering the certificate data or printed numerical tables.
The repository's \href{https://github.com/SelenaJXWANG/srg-gegenbauer-counting/tree/main/paper}{\texttt{paper/} directory}
contains the article and Supplement~\cite{supp405}, in PDF and TeX.
The \href{https://github.com/SelenaJXWANG/srg-gegenbauer-counting/blob/85cd39b74c9ef8a61ab746deb1102af3eeae7e10/extras/certificate_inputs.sha256}{fixed-input manifest}
\nolinkurl{extras/certificate_inputs.sha256} lists the mathematical
inputs and exact verification engines. Its SHA-256 is
\begin{center}\small\ttfamily
58ebb057e5d840a1571d503bf2c3a4fc0d1854fe6c7336615eb1fc7bc0398b9b
\end{center}
The manifest excludes the manuscript and generated reports, so this
identifier does not depend on the document that prints it. The complete
release file hashes, including the papers, are in
\nolinkurl{extras/SHA256SUMS}.

Table~\ref{tab:certificate-files} locates the certificate used for each
parameter set. A case directory is
\nolinkurl{results/certificates/v_k_lambda_mu/}, with the four numbers
substituted for its name. The exact paths, supporting counting and
structural source files, and the corresponding verification engines
are indexed in
\href{https://github.com/SelenaJXWANG/srg-gegenbauer-counting/blob/85cd39b74c9ef8a61ab746deb1102af3eeae7e10/extras/certificate_index.json}{\nolinkurl{extras/certificate_index.json}}.
The complete data are in the
\href{https://github.com/SelenaJXWANG/srg-gegenbauer-counting/tree/85cd39b74c9ef8a61ab746deb1102af3eeae7e10/results/certificates}{\texttt{results/certificates/} directory}.
\begin{table}[H]
\centering\small
\caption{Certificate inputs used for the displayed contradictions.
Filenames are relative to the indicated case directory; the four
five-point certificates are read directly from the main TeX tables.}
\label{tab:certificate-files}
\begin{tabular}{@{}lp{.64\textwidth}@{}}
\toprule
$(v,k,\lambda,\mu)$ & Certificate source\\
\midrule
$(351,140,73,44)$ & \nolinkurl{raw_certificate.json}\\
$(375,102,45,21)$ & \nolinkurl{exact_certificate.json}\\
$(405,132,63,33)$ & \nolinkurl{ordinary_counting_certificate/exact_certificate.json}\\
$(441,88,35,13)$ & \nolinkurl{exact_certificate.json}\\
$(476,133,60,28)$ & \nolinkurl{raw_certificate.json}\\
$(540,147,66,30)$ & Tables~\ref{tab:direct-functions} and~\ref{tab:direct-coeff-a}\\
$(550,162,75,36)$ & Tables~\ref{tab:direct-functions} and~\ref{tab:direct-coeff-b}\\
$(575,112,45,16)$ & \nolinkurl{exact_certificate.json}\\
$(703,182,81,35)$ & Tables~\ref{tab:direct-functions} and~\ref{tab:direct-coeff-b}\\
$(1344,221,88,26)$ & Tables~\ref{tab:direct-functions} and~\ref{tab:direct-coeff-a}\\
\bottomrule
\end{tabular}
\end{table}
To reproduce the fixed snapshot, run
\begingroup\small
\begin{verbatim}
git clone https://github.com/SelenaJXWANG/srg-gegenbauer-counting.git
cd srg-gegenbauer-counting
git checkout 85cd39b74c9ef8a61ab746deb1102af3eeae7e10
python3 -S code/verify.py
\end{verbatim}
\endgroup
This requires Python 3.9 or later and only its standard library.
The five stages check the retained ten-case constructions, the ordinary
counting proof and all-degree comparison for $405$, the actual short
TeX certificates, the expanded TeX certificates, and the article's
parameter and structural data. A successful run ends with
\begin{quote}\small\ttfamily
PASS: five exact verification stages; fixed inputs unchanged.
\end{quote}
The generated report is \nolinkurl{results/latest_verification/summary.json};
the retained check for this release is
\nolinkurl{results/release_checks/summary.json}. No SDP solver is needed.
The common command also stages the preserved engines with their required
input paths; individual engine files are not separate public entry points.

The same command is run by the repository's
\href{https://github.com/SelenaJXWANG/srg-gegenbauer-counting/actions}{GitHub Actions workflow}.
For the fixed snapshot above, it completed successfully under Python
3.9 and 3.12; the
\href{https://github.com/SelenaJXWANG/srg-gegenbauer-counting/actions/runs/38571633757}{recorded workflow run}
provides the execution logs. These checks verify the finite rational
calculations, while the counting and positivity arguments in the article
justify the mathematical implications. Code is licensed under MIT;
the article, supplement, and certificate data are licensed under CC BY 4.0.

Where earlier raw JSON certificates are retained, they need not match
the printed five-point directions. The TeX tables are
the inputs for those four printed proofs, and $1344$ also has an
independent reconstruction in
\nolinkurl{code/verification/verify_1344_certificate.py}.
For $405$, \nolinkurl{ordinary_counting_certificate/} is the present
proof; \nolinkurl{earlier_structural_certificate/} is a distinct retained
proof using extra zero-count relations. Optional numerical searches
may produce further certificates and are separate from this replay.

\subsection{Exact data and graph encodings}
The supplementary data contain all rational source rows, directions,
weights, coordinate substitutions and verification programs. The
verifiers reconstruct graph types, count substitutions and kernel
coefficients, check reference Gram matrices by exact elimination, and
check each final rational sign. They require no numerical optimizer.
Appendices~\ref{app:directions} and~\ref{app:weights} print the functions
and multipliers for $351$ and $476$; Supplement~\cite{supp405}
prints the complete $405$ certificate and its reconstruction formulas.
The full numerical source lists for the structural cases accompany the
article; Section~\ref{sec:structural} proves their source conditions and
gives the resulting explicit contradictions.

For the five-point and six-point certificate data,
edge strings use the order
$01,02,\ldots,0(h-1),12,\ldots,(h-2)(h-1)$.
For the structural and $405$ data, an integer graph mask puts edge $ij$,
$i<j$, in bit $j(j-1)/2+i$, giving the order
$01,02,12,03,13,23,\ldots$.
In the structural data, a catalogue record $(\mathrm{id},h,b)$
identifies $\mathrm{id}$ with the $h$-vertex graph of mask $b$;
catalogue identifiers are not masks. The graphs $C_j$ follow the
free-column order recorded with the affine substitution.
All profile vectors use binary lexicographic order with coordinate $0$
first. The decoding convention is part of the certificate specification.

\subsection{Normalization of the structural certificates}
\label{app:structural-normalization}
For an $m$-vertex anchor, put $D=d-m$ and
$u_x=\sqrt{3-\rho}\,p_x$. Write
$\zeta=\langle u_x^\perp,u_y^\perp\rangle$,
$\eta=\|u_x^\perp\|^2$ and $\eta'=\|u_y^\perp\|^2$.
The structural source data use the kernels
$1$, $\zeta$ and $\zeta^2-\eta\eta'/D$.
Consequently a source row before its positive row normalization is
\[
 -\frac{2^{\binom{m+1}{2}}}{(v)_{m+1}}\,
 \chi_\ell S(H,\ell,f),\qquad
 (\chi_0,\chi_1,\chi_2)=
 \left(1,3-\rho,(3-\rho)^2\frac{D-1}{D}\right).
\]
The $R_i$ in \eqref{eq:struct-margin} are the resulting rows after
the stored positive row normalization.
All reference Gram matrices are positive definite. The stored $f$ gives
rational values on the retained profiles in binary lexicographic order.
A profile is omitted only if its one-vertex extension has a proved zero
count or its prescribed residual squared length is negative; neither
can occur at an actual reference copy. Set $f=0$ on every such profile.
Coefficients at forbidden graph counts are removed by their zero-count
equations. These conventions introduce no additional sign or scale.

For each parameter set, the data contain $58$ integer witnesses for
zero counts and $273$ witnesses for $76$ distinct completion equations.
After ordinary counting substitutions their rank is $63$. The exact
affine map has $110$ free scaled graph counts and constant coordinate $1$;
its free rows are an identity matrix. Exact substitution and the rank
calculation show that it parametrizes the entire solution space of
these equations. The rational source lists, row normalizations and
positive weights determine every coefficient in
\eqref{eq:struct-margin}, including the bounded residual terms.

\begin{table}[H]
\centering
\caption{Sources in the three structural certificates, before adding the
count-bound terms of~\eqref{eq:struct-margin}.}
\label{tab:struct-certificates}
\begin{tabular}{rrrrr}
\toprule
$v$&Count sources&Kernel sources&Quartics&Total\\
\midrule
375&13&97&1&111\\
441&21&88&2&111\\
575&60&48&3&111\\
\bottomrule
\end{tabular}
\end{table}
\subsection{Five-point and six-point data normalization}
The retained five-point and six-point data store the counts $N_F/(v|\Aut F|)$ and the energies
$S(H,\ell,f)/(v|\Aut H|)$. Multiplying an archived identity by $v$ restores
the present normalization, with automorphism factors absorbed into the
multipliers. The definition of $S$ is unchanged.

\subsection{A separate numerical reproduction}
The accompanying reproduction code implements a search for fresh
rational certificates. In a separate research run, it produced verified
certificates for all ten parameter sets. The six cases of
orders five and six are rebuilt from the parameters; the four seven-point
cases use fresh optimization on the previously verified counting and
structural models. Their rational certificates and conversions to the
present $S$ convention pass exact verification. They are alternative
positive combinations; all displayed functions, source inventories and
margins in this paper refer to the certificates specified here.
The distribution includes these printed certificates and the discovery
code; the alternative numerical-run outputs are not needed for their
verification and are not included.

The numerical reproduction uses CVXPY~\cite{cvxpy} to model the
semidefinite programs and Clarabel~\cite{clarabel} as its default solver.
Array operations use NumPy~\cite{numpy}, and numerical linear programs
use SciPy~\cite{scipy} through its HiGHS interface. Exact replay uses
Python's rational arithmetic and does not depend on a numerical solver.

\section{Coefficients for the five-point certificates}\label{app:five-coefficients}
The entries below are the strictly positive rational coefficients in
\eqref{eq:direct-contradiction}. Together with the functions and multipliers
in Table~\ref{tab:direct-functions}, they give every coefficient of the
four identities in Section~\ref{sec:direct}.
\begin{table}[H]\centering\small
\caption{Exact coefficients in \eqref{eq:direct-contradiction}, using labelled counts $N_F$.}
\label{tab:direct-coeff-a}
\renewcommand{\arraystretch}{1.8}
\begin{tabular}{ccc}\toprule
&$v=540$&$v=1344$\\\midrule
$b_0$&$\dfrac{574630684892568}{6496825}$&$\dfrac{11861558480040192}{18036575}$\\
$b_1$&$\dfrac{519736215717}{124791014600000}$&$\dfrac{5233918851}{171401572225}$\\
$b_2$&$\dfrac{66130353}{909555500}$&$\dfrac{123153959}{15944332300}$\\
$b_3$&$\dfrac{2901906377541}{62395507300000}$&$\dfrac{22833078929}{1028409433350}$\\
$b_4$&$\dfrac{341895786273}{24958202920000}$&$\dfrac{118254493}{16587248925}$\\
\bottomrule\end{tabular}\end{table}
\begin{table}[H]\centering\small
\caption{Exact coefficients in \eqref{eq:direct-contradiction}, using labelled counts $N_F$.}
\label{tab:direct-coeff-b}
\renewcommand{\arraystretch}{1.8}
\begin{tabular}{ccc}\toprule
&$v=550$&$v=703$\\\midrule
$b_0$&$\dfrac{3153820063124750}{4826151}$&$\dfrac{827108257923964177}{2718487200}$\\
$b_1$&$\dfrac{2768645749393}{8561109258900}$&$\dfrac{37006789818385019}{667991024363520000}$\\
$b_2$&$\dfrac{19500242539}{15732075150}$&$\dfrac{133712632548389}{343508042592000}$\\
$b_3$&$\dfrac{296782302161}{23543050461975}$&$\dfrac{665498801527621007}{8683883316725760000}$\\
$b_4$&$\dfrac{1327509707}{778282659900}$&$\dfrac{509997958311169}{108548541459072000}$\\
\bottomrule\end{tabular}\end{table}
\section{Rational directions for the six-point certificates}\label{app:directions}
The tables give all rational directions for $351$ and $476$.  For $m=4$, the sixteen
entries are in binary lexicographic order $0000,0001,\ldots,1111$;
edge strings list $01,02,03,12,13,23$.  As before, the vector in a row
is a complete specification of the function on adjacency sets.

\begingroup\small\begin{longtable}{rrrl}
\caption{All rational directions for $v=351$; divide each displayed vector by $30$.}\label{tab:direct-directions-351}\\
\toprule $i$&$H_i$&$\ell_i$&$30f_i$\\\midrule\endfirsthead
\toprule $i$&$H_i$&$\ell_i$&$30f_i$\\\midrule\endhead
\bottomrule\endfoot
1&000&1&\scriptsize $(1,9,9,-20,9,-20,-20,30)$\\
2&000000&1&\scriptsize $(0,3,-5,7,-3,0,27,-30,5,-27,0,18,-7,29,-17,0)$\\
3&000001&1&\scriptsize $(0,-7,-7,-5,-9,5,5,16,-9,5,5,16,15,-1,-1,-30)$\\
4&000011&1&\scriptsize $(-2,4,-9,-7,-9,-7,25,5,1,-20,12,20,12,20,-30,-17)$\\
5&000011&1&\scriptsize $(0,0,11,-18,-11,18,0,0,0,0,-24,30,24,-30,0,0)$\\
6&000011&2&\scriptsize $(5,-18,5,-6,5,-6,9,8,4,16,-12,12,-12,12,-30,9)$\\
7&000111&1&\scriptsize $(5,-5,-5,-6,-5,-6,-6,29,-24,17,17,1,17,1,1,-30)$\\
8&001011&1&\scriptsize $(8,-30,4,15,4,15,-8,-9,4,15,-8,-9,-8,-9,3,12)$\\
9&001100&1&\scriptsize $(0,6,-6,0,-6,0,30,-24,6,-30,0,24,0,24,-24,0)$\\
10&001100&1&\scriptsize $(-6,9,9,-6,9,-6,-1,-9,9,-1,-6,-9,-6,-9,-9,30)$\\
11&001100&2&\scriptsize $(0,18,-15,-1,15,-14,0,-30,-18,0,14,25,1,-25,30,0)$\\
12&001111&1&\scriptsize $(10,-19,-5,16,-5,16,-13,-1,-5,13,-7,-4,-7,-4,30,-16)$\\
13&011110&1&\scriptsize $(15,-8,-8,-2,-8,-1,-1,14,-8,-1,-1,14,-2,14,14,-30)$\\
14&011110&2&\scriptsize $(3,-6,-6,30,-6,-19,-19,4,-6,-19,-19,4,30,4,4,19)$\\
15&011111&1&\scriptsize $(1,-12,-12,30,0,3,3,-13,0,3,3,-13,7,-2,-2,5)$\\
16&011111&1&\scriptsize $(-15,7,7,-3,21,-1,-1,-14,21,-1,-1,-14,-30,-2,-2,30)$\\
17&111111&1&\scriptsize $(16,-12,-12,3,-12,3,3,11,-12,3,3,11,3,11,11,-30)$\\
\end{longtable}\endgroup

\begingroup\small\begin{longtable}{rrrl}
\caption{All rational directions for $v=476$; divide each displayed vector by $10$.}\label{tab:direct-directions-476}\\
\toprule $i$&$H_i$&$\ell_i$&$10f_i$\\\midrule\endfirsthead
\toprule $i$&$H_i$&$\ell_i$&$10f_i$\\\midrule\endhead
\bottomrule\endfoot
1&111&1&\scriptsize $(-2,5,5,-7,5,-7,-7,10)$\\
2&000011&1&\scriptsize $(0,0,1,-4,-1,4,0,0,0,0,-7,10,7,-10,0,0)$\\
3&001011&1&\scriptsize $(0,0,4,-1,-2,1,-6,4,-2,1,-4,3,10,-7,0,0)$\\
4&001011&1&\scriptsize $(4,-10,-1,7,-1,7,0,-4,-1,7,0,-4,0,-4,0,3)$\\
5&001100&1&\scriptsize $(0,-1,1,0,1,0,9,-10,-1,-9,0,10,0,10,-10,0)$\\
6&001111&1&\scriptsize $(0,-9,-4,10,-4,10,-4,0,-4,7,8,-9,8,-9,-2,1)$\\
7&011110&1&\scriptsize $(0,9,9,1,-9,0,0,-10,-9,0,0,-10,-1,10,10,0)$\\
8&011110&1&\scriptsize $(0,1,-1,0,-1,-1,10,-8,1,-10,1,8,0,10,-10,0)$\\
9&011110&2&\scriptsize $(0,0,0,-10,0,0,0,-1,0,0,0,-1,10,1,1,0)$\\
10&011110&2&\scriptsize $(0,-6,6,0,10,-2,-7,7,-10,7,2,-7,0,-4,4,0)$\\
11&011111&1&\scriptsize $(1,0,0,1,-3,-1,-1,3,-3,-1,-1,3,2,5,5,-10)$\\
12&011111&1&\scriptsize $(0,0,0,0,0,-5,5,0,0,-5,5,0,0,10,-10,0)$\\
13&011111&1&\scriptsize $(0,-4,-4,10,-2,1,1,-3,-2,1,1,-3,8,-3,-3,-1)$\\
14&011111&2&\scriptsize $(0,0,0,0,-6,3,3,10,6,-3,-3,-10,0,0,0,0)$\\
15&111111&1&\scriptsize $(3,7,7,-10,7,-10,-10,7,7,-10,-10,7,-10,7,7,1)$\\
16&111111&2&\scriptsize $(2,-9,-9,0,-9,0,0,10,-9,0,0,10,0,10,10,-3)$\\
\end{longtable}\endgroup

\section{Integer multipliers for the six-point certificates}\label{app:weights}
\providecommand{\srgBigInt}[3]{\left[\begin{array}{@{}r@{}}#1\\#2\\#3\end{array}\right]_{40}}
The multipliers include all automorphism factors and refer to labelled
$N$ and unnormalized $S$. There are $17$ kernel terms and $7$ count
terms for $351$, and $16$ kernel terms and $8$ count terms for $476$.
With the normalization of Section~\ref{sec:direct}, each identity has
constant $-1$. The machine-readable counterparts are the two
\nolinkurl{raw_certificate.json} files in Table~\ref{tab:certificate-files};
the verifier compares them with every integer printed below. The full
tables are retained so that the exact certificate is also specified
within the paper. Within each $v$-block write
$\omega_i=\omega_{v,i}$ and $\zeta_j=\zeta_{v,j}$. We print integers as
\[
 \srgBigInt{a}{b}{c}=a\,10^{80}+b\,10^{40}+c
 \qquad(0\le a,b,c<10^{40}).
\]
Graphs use the edge-string convention in Appendix~\ref{app:verification}.
The common scales are
\[
 L_{351}=\srgBigInt{102711949647804269727238704359485866911}{7090228324872847275751373679454059820670}{1076412292550227782101992011071927842880}.
\]
\[
 L_{476}=\srgBigInt{128341483}{4387833639899167631266490996799829173466}{4104560806157292825403747795632063212800}.
\]
\clearpage
\begingroup\fontsize{7.2}{9}\selectfont
\setlength{\tabcolsep}{1.5pt}
\renewcommand{\arraystretch}{1.06}
\begin{longtable}{rcrlc}
\caption{Complete integer weights for $v=351$. The row $i$ selects the function in Appendix~\ref{app:directions}; $G_j$ is the six-vertex graph with the displayed edge string.}\label{tab:six351-weights}\\
\toprule $i$&$\omega_i$&$j$&$G_j$&$\zeta_j$\\\midrule\endfirsthead
\toprule $i$&$\omega_i$&$j$&$G_j$&$\zeta_j$\\\midrule\endhead
\bottomrule\endfoot

1&$\srgBigInt{140262026891169625080174076332493}{402642432273172888010149145022689052449}{9108507321249499450249403888842809600000}$&1&\texttt{000010001001011}&$\srgBigInt{21727122221693832488348818346229}{5870717876111924047234281078378676719478}{4202236583231146898001697016679035516330}$\\[1pt]

2&$\srgBigInt{9027285625014494631942984406547}{2285407457457158092134549667463318507240}{5528233941778865256370366886213792456000}$&2&\texttt{000110011011110}&$\srgBigInt{454817647483139570977157992483908}{4525746483369943415330728387435296123075}{2079341097724794930652082584472056188819}$\\[1pt]

3&$\srgBigInt{188304312248630294141452859386598}{2427829677299695212762874421695603519227}{8802927176755503307415108560935126400000}$&3&\texttt{000110011011111}&$\srgBigInt{347683263903066310248246571598587}{4685715507194972718252169150185073625732}{1916025133489648510404303567469403433546}$\\[1pt]

4&$\srgBigInt{33007457788580689459563718579015}{5196092378025351170368317686296951450104}{4684986855306800462739663305789919500000}$&4&\texttt{000110011111110}&$\srgBigInt{3369959043029563906427789218335620}{4530034752955328546104660808623111565624}{6026712397667319220082215224249717034184}$\\[1pt]

5&$\srgBigInt{6017719748416982696758134552982}{1871196422489024955017279954013811713578}{6277584542500652125291651896329654700000}$&5&\texttt{000110111111000}&$\srgBigInt{609026391690289958078804084535506}{7304131907204129798061187119586083204943}{245263261426625719659511163466377067579}$\\[1pt]

6&$\srgBigInt{97773520988983403406256931285066}{1835992994959336292071600822142713052862}{9833052338791783133787084248417162500000}$&6&\texttt{001110111111000}&$\srgBigInt{116307044230683818208992812601676}{5068737547707512623155175408717750511277}{8078259045736268576347343779774397582004}$\\[1pt]

7&$\srgBigInt{35400097741269002402380729786438}{3289049985537485265485691822790714245864}{8012796245313801796832855062677630576000}$&7&\texttt{001110111111001}&$\srgBigInt{1444815902054691834456912919554414}{5852692936275325101946402549302898575841}{5895167162896407063830271323309577758504}$\\[1pt]

8&$\srgBigInt{369509214537610886754996110746551}{2147920939113707467005112703494131646343}{6202941347112613832159927597385497344000}$&&&\\[1pt]

9&$\srgBigInt{61640744961074394115235478358136}{4790871484077601812963527361987060278298}{7401938012689846937072584160480639872000}$&&&\\[1pt]

10&$\srgBigInt{231445612675255990474046680219498}{2965192612214787131889538087854222302562}{3976570132192282071534601936029325184000}$&&&\\[1pt]

11&$\srgBigInt{14484213607788919130425343915937}{3148873514931183152823666235384313825288}{2037616496668923533911873336370135040000}$&&&\\[1pt]

12&$\srgBigInt{899089634309870492770800215884828}{249607215509471011429960060384782485374}{5770175143889076782026056863049010480000}$&&&\\[1pt]

13&$\srgBigInt{161954488799111804444772990608552}{1028559918805719878034494956705656961357}{3317545697735941650026783493441282780000}$&&&\\[1pt]

14&$\srgBigInt{413358868328204583373213681940387}{9866345230542660047869583168615358465701}{9313466666083494675981761391317817297500}$&&&\\[1pt]

15&$\srgBigInt{73156995329616085324056788660501}{975747532324816598714120615853914249585}{5579840156451171999389650347285131300000}$&&&\\[1pt]

16&$\srgBigInt{306359992133369510665913881378707}{7198911732455708046783206681994856035895}{6309382295627641450303190210818896800000}$&&&\\[1pt]

17&$\srgBigInt{39087066120075618692186837051389}{3544778065860184420249950747432742974112}{1107790338959812320089453460505268392000}$&&&\\[1pt]

\end{longtable}\endgroup
\clearpage
\begingroup\fontsize{7.2}{9}\selectfont
\setlength{\tabcolsep}{1.5pt}
\renewcommand{\arraystretch}{1.06}
\begin{longtable}{rcrlc}
\caption{Complete integer weights for $v=476$. The row $i$ selects the function in Appendix~\ref{app:directions}; $G_j$ is the six-vertex graph with the displayed edge string.}\label{tab:six476-weights}\\
\toprule $i$&$\omega_i$&$j$&$G_j$&$\zeta_j$\\\midrule\endfirsthead
\toprule $i$&$\omega_i$&$j$&$G_j$&$\zeta_j$\\\midrule\endhead
\bottomrule\endfoot

1&$\srgBigInt{1}{1532903875339920837204373990570212030453}{7549696509976445269427155079054519695200}$&1&\texttt{000010011011110}&$\srgBigInt{26}{9830849083055629150240368678627677481105}{9393004220420023157014711307998623552367}$\\[1pt]

2&$\srgBigInt{5}{4211636264983894030186443184728863211590}{234861464398603048720636543520782956800}$&2&\texttt{000110011011111}&$\srgBigInt{95}{7083599885321012957152103625761292689695}{4674580075927447635648727252566658637768}$\\[1pt]

3&$\srgBigInt{8}{770386311975691435167761473344271024643}{6576510960066313837102290165927183884000}$&3&\texttt{000110011111110}&$\srgBigInt{51}{422891615578313615850638467025161051840}{5110526244023460081340274099125868909762}$\\[1pt]

4&$\srgBigInt{9}{3946670426081474522384315251415560012652}{2689884726462183010389663315410250798000}$&4&\texttt{000110111111000}&$\srgBigInt{74}{7250411576366376348243592631984137955526}{8628969446137315823148429999298070496028}$\\[1pt]

5&$\srgBigInt{3}{615785654518659065687888819388354524540}{6185773042805127299448341693985913240000}$&5&\texttt{000110111111010}&$\srgBigInt{32}{9879573890969416942948506319447406128698}{6956066889439371456165445230857791365343}$\\[1pt]

6&$\srgBigInt{24}{9006202787169927901301480772000041371074}{8797589647922124754375056207932591616000}$&6&\texttt{001110111111000}&$\srgBigInt{126}{2603439299125317564430478465934637535447}{248053378819080341004013058929739958835}$\\[1pt]

7&$\srgBigInt{1}{690707263130732125622572270541627184547}{6139204641884343984771427977253606716250}$&7&\texttt{001110111111011}&$\srgBigInt{169}{3175493685256253313785902097733855861118}{875183165568377264664111504225650974669}$\\[1pt]

8&$\srgBigInt{2}{1559687142264927778101806481602550726851}{8296139835015759052404005500241945691800}$&8&\texttt{011111111011110}&$\srgBigInt{5}{4972661694553314801873304817326813829399}{1871466464399734527480048897333033374624}$\\[1pt]

9&$\srgBigInt{71}{1533514168334911183500195002955778876296}{5119448340281984647164750642267255840325}$&&&\\[1pt]

10&$\srgBigInt{37}{1714642777539251493676478363706140175326}{9336930383837753886146814148139860589100}$&&&\\[1pt]

11&$\srgBigInt{7}{4457612952471220848564583332389729623008}{8077269000384674118960023021587680383600}$&&&\\[1pt]

12&$\srgBigInt{23}{463295452950103318511592310413201110236}{8238611539394202537707693911158640337496}$&&&\\[1pt]

13&$\srgBigInt{120}{7576419446140606090820643123118388303458}{266298012796449102025379519228211453200}$&&&\\[1pt]

14&$\srgBigInt{1}{701463304004479506919490660943427723153}{2771639061770650087630270675737582505600}$&&&\\[1pt]

15&$\srgBigInt{2}{5447305043906599301240013889074606104941}{1117852532327002617425936355666249410400}$&&&\\[1pt]

16&$\srgBigInt{4}{3008436872385535446416295472525510036297}{7439257850852303756512661288968216684800}$&&&\\[1pt]

\end{longtable}\endgroup
\clearpage
\section*{Acknowledgements}
The mathematical ideas, research design, choice of constraints, and
interpretation of the results were developed by the authors. We used
GPT (OpenAI) to assist with implementing these ideas in code and
revising the implementations under our direction. The resulting
nonexistence claims are supported by the exact rational certificates
and verification procedures supplied with this article. The authors
take responsibility for the mathematical arguments, computations,
and presentation.

\end{document}